\documentclass[letterpaper, 10 pt, conference]{ieeeconf}

\IEEEoverridecommandlockouts                              
\usepackage{graphicx}
\usepackage{amsmath}
\usepackage{amssymb}
\usepackage{amsfonts}
\usepackage{latexsym}
\usepackage{epsfig}
\usepackage{theorem}
\usepackage{fancybox}
\usepackage{shadow}
\usepackage{float}
\usepackage{comment}
\usepackage{color}
\usepackage{enumerate}
\usepackage{psfrag}
\usepackage{empheq}
\usepackage{cite}
\usepackage{epstopdf}

\def \RR {{\mathbb R}}
\def \NN {{\mathbb N}}

\def \eps {\varepsilon}

\newtheorem{propo}{Proposition}
\newtheorem{lem}{\bf{Lemma}}
\newtheorem{theo}{Theorem}
\newtheorem{coro}{Corollary}
\newtheorem{rem}{Remark}

\newtheorem{hypo}{\bf{Assumption}}

\begin{document}
\title{\bf Gradient-like observer design on the Special Euclidean group SE(3) with system outputs on the real projective space}

\author{Minh-Duc Hua, Tarek Hamel, Robert Mahony, Jochen Trumpf
\thanks{Minh-Duc Hua is with ISIR UPMC-CNRS (Institute for Intelligent Systems and Robotics), Paris, France. E-mail: {\tt\footnotesize $hua@isir.upmc.fr$}.}
\thanks{Tarek Hamel is with I3S UNS-CNRS, Sophia Antipolis, France. E-mail: {\tt\footnotesize $thamel@i3s.unice.fr$}.}%
\thanks{Robert Mahony and Jochen Trumpf are with the Research School of Engineering,
Australian National University, Canberra, Australia. E-mails: {\tt\footnotesize $Robert.Mahony(Jochen.Trumpf)@anu.edu.au$}.}
}

\maketitle

\pagestyle{empty}
\thispagestyle{empty}

\begin{abstract}
A nonlinear observer on the Special Euclidean group $\mathrm{SE(3)}$ for full pose estimation, that takes the system outputs on the real projective space directly as inputs, is proposed. The observer derivation is based on a recent advanced theory on nonlinear observer design. A key advantage with respect to existing pose observers on $\mathrm{SE(3)}$ is that we can now incorporate in a unique observer different types of measurements such as vectorial measurements of known inertial vectors and position measurements of known feature points. The proposed observer is extended allowing for the compensation of unknown constant bias present in the velocity measurements. Rigorous stability analyses are equally provided. Excellent performance of the proposed observers are shown by means of simulations.
\end{abstract}

\section{Introduction}\label{sec:introduction}

The development of a robust and reliable estimator of the pose (i.e. position and attitude) of a rigid body is a key requirement for robust and high performance control of robotic vehicles. Pose estimation is a highly nonlinear problem in which the sensors normally utilized are prone to non-Gaussian noise \cite{Baldwin}. Classical approaches for state estimation are based on nonlinear filtering techniques such as extended Kalman filters, unscented Kalman filters or particle filters. However, nonlinear observers have become an alternative to these classical techniques, starting with the work of Salcudean \cite{Salcudean91} for attitude estimation and subsequent contributions over the last two decades \cite{nf99,vf01,rg03,ts03,tmrm07,mhp08,bmtICRA09,vcso10, ms10cep, rt11CDC,hamelCDC11,grip12,huaCDC09,hua14,huaCDC14,eudes2013,trumpf12}. Early nonlinear attitude observers have been developed on the basis of Lyapunov analysis. Recently, the attitude estimation problem has motivated the development of theories on invariant observers for systems endowed with symmetry properties \cite{ar03,bmr08, bmr09,mhp08,ltm10,trumpf12,MahonyNOLCOS13,khosravian15}. For instance, complementary nonlinear attitude observers exploiting the underlying Lie group structure of the Special Orthogonal group $\mathrm{SO(3)}$ are derived in \cite{mhp08} with proofs of almost global stability of the error system. A symmetry-preserving nonlinear observer design based on the Cartan moving-frame method is proposed in \cite{bmr08,bmr09}, which is locally valid for arbitrary Lie groups. A gradient-like observer design technique for invariant systems on Lie groups is proposed in \cite{ltm10}, leading to almost global convergence provided that a non-degenerate Morse-Bott cost function is used. More recently, an observer design method directly on the homogeneous output space for the kinematics of mechanical systems is proposed in \cite{MahonyNOLCOS13}, leading to autonomous error evolution and strong convergence properties. Finally, \cite{khosravian15} extends the observer design methodology proposed in \cite{MahonyNOLCOS13} in order to deal with the case where the measurement of system input is corrupted by an unknown constant bias.

Full pose observer design, although less studied than attitude observer design, has recently attracted  more attention \cite{vf01,rg03,Baldwin,bmthcECC07,bmtICRA09,vcso10,hua11SE3}. For instance, observers designed directly on $\mathrm{SE(3)}$ have been proposed using both full state feedback \cite{bmthcECC07} or bearing measurements of known landmarks \cite{bmtICRA09}. An observer on $\mathrm{SO(3)} \times \RR^3$ is proposed in \cite{vcso10}, using full range and bearing measurements of known landmarks and achieving almost global asymptotic stability. In a prior work by the authors \cite{hua11SE3}, a nonlinear observer on $\mathrm{SE(3)}$ is proposed using directly position measurements in the body-fixed frame of known inertial feature points or landmarks, with motivation strongly related to robotic vision applications using either stereo camera or Kinect sensor. The observer derivation is based on the gradient-like observer design technique proposed in \cite{ltm10}, and the almost global asymptotic stability of the error system is proved by means of Lyapunov analysis.

In this paper, we consider the question of deriving a nonlinear observer on $\mathrm{SE(3)}$ for full pose estimation that takes the system outputs on the real projective space $\mathbb{RP}^3$ directly as inputs. A key advance on our prior work \cite{hua11SE3} is the possibility of incorporating ``naturally'' in a sole observer both vectorial measurements (provided e.g. by magnetometers or inclinometers) and position measurements of known inertial feature points (provided e.g. by stereo camera). In addition, sharing the same robustness property with the observer proposed in \cite{hua11SE3}, the  algorithm here proposed is also well-posed even when there is insufficient data for full pose reconstruction using algebraic techniques. In such situations, the proposed observer continues to operate, incorporating what information is available and relying on propagation of prior estimates where necessary. Finally, as a complementary contribution, a modified version of the basic observer is proposed so as to deal with the case where bias is present in the velocity measurements.

The remainder of this paper is organised as follows. Section \ref{sec:preliminary} formally introduces the problem of pose estimation on $\mathrm{SE(3)}$ along with the notation used. In Section III, based on a recent advanced theory for nonlinear observer design directly on the output space \cite{MahonyNOLCOS13}, a nonlinear observer on $\mathrm{SE(3)}$ is proposed using direct body-fixed measurements of known inertial elements of the real projective space $\mathbb{RP}^3$ and the knowledge of the group velocity. Stability analysis is also provided in this section. Then, in Section \ref{sec:observerDesignBias} the proposed basic observer is extended using Lyapunov theory in order to cope with the case where the measurement data of the group velocity are corrupted by an unknown constant bias. In Section \ref{sec:simulation}, the performance of the proposed observers are validated by means of simulation. Finally, concluding remarks are given in Section \ref{sec:conclusions}.

\section{Preliminary material}\label{sec:preliminary}
\subsection{Notation}

Let $\{{\cal A}\}$ and $\{{\cal B}\}$ denote an inertial frame and a body-fixed frame attached to a vehicle moving in 3D-space, respectively. The vehicle's position, expressed in the frame $\{{\cal A}\}$, is denoted as $p\in \RR^3$. The attitude of the vehicle is represented by a rotation matrix $R \in \mathrm{SO(3)}$ of the frame $\{{\cal B}\}$ relative to the frame $\{{\cal A}\}$. Let $V \in \RR^3$ and $\Omega \in \RR^3$ denote the vehicle's translational and angular velocities, both expressed in $\{{\cal B}\}$.

In this paper, we consider the problem of estimating the vehicle's pose, which can be represented by an element of the Special Euclidean group $\mathrm{SE(3)}$ given by the matrix
\begin{equation}\label{defpose}
X := \begin{bmatrix} R & p\\ 0 & 1\end{bmatrix} \in \mathrm{SE(3)} \subset \RR^{4\times 4}.
\end{equation}
This representation, known as homogeneous coordinates, preserves the group structure of $\mathrm{SE(3)}$ with the $\mathrm{GL(4)}$ operation of matrix multiplication, i.e. $X_1 X_2 \in \mathrm{SE(3)}$, $\forall X_1,X_2 \in \mathrm{SE(3)}$. Now let us recall some common definitions and notation.

\noindent $\bullet$ The Lie-algebra $\mathfrak{se}(3)$ of the group $\text{SE(3)}$ is defined as
\[
\mathfrak{se}(3) := \left\{A \in \RR^{4\times 4} \mid \exists\, \Omega, V \in \RR^3\, : \,A = \begin{bmatrix} \Omega_\times & V\\ 0 & 0\end{bmatrix} \right\},
\]
with $\Omega_\times$ denoting the skew-symmetric matrix associated with the cross product by $\Omega$,  i.e. $\Omega_\times v = \Omega \times v$, $\forall v \in \RR^3$.
The adjoint operator is a mapping $Ad: \mathrm{SE(3)} \times \mathfrak{se}(3) \rightarrow \mathfrak{se}(3)$ defined as
$Ad_{X}A := X A X^{-1}$, with $X\in \mathrm{SE(3)}, A \in \mathfrak{se}(3)$.

\noindent $\bullet$ For any two matrices $M_1, M_2 \in \RR^{n\times n}$, the Euclidean matrix inner product and Frobenius norm are defined as
\[
\langle M_1, M_2 \rangle := {\mathrm{tr}}(M_1^\top M_2),\,\,\, \| M_1 \| := \sqrt{\langle M_1, M_1 \rangle}.
\]
Let $\mathbf{P}_a(M)$, $\forall M \in \RR^{n\times n}$, denote the anti-symmetric part of $M$, i.e. $\mathbf{P}_a(M) := (M-M^\top)/2$.
Let $\mathbf{P} : \RR^{4\times 4} \rightarrow \mathfrak{se}(3)$ denote the unique orthogonal projection of $\RR^{4\times 4}$ onto $\mathfrak{se}(3)$ with respect to the inner product $\langle \cdot, \cdot \rangle$, i.e. $\forall A \in \mathfrak{se}(3)$, $M \in \RR^{4\times 4}$, one has $$\langle A, M \rangle = \langle A, \mathbf{P}(M) \rangle = \langle \mathbf{P}(M),A \rangle\,.$$
It is verified that for all $M_1 \in \RR^{3\times 3}, m_{2,3} \in \RR^3, m_4 \in \RR$,
\begin{equation}\label{projection}
\mathbf{P}\left(\begin{bmatrix} M_1 & m_2\\ m_3^\top & m_4\end{bmatrix}\right) := \begin{bmatrix} \mathbf{P}_a(M_1) & m_2\\ 0 & 0 \end{bmatrix}.
\end{equation}

\noindent $\bullet$ For all $X \in \mathrm{SE(3)}$, $A_1, A_2 \in \mathfrak{se}(3)$, the following equation defines a {{\em right-invariant Riemannian metric}} $\langle\cdot,\cdot\rangle_X$:
\[
\langle A_1 X, A_2 X \rangle_X := \langle A_1, A_2 \rangle .
\]

\noindent $\bullet$ For any $x \in \RR^4$ (or $\in \mathbb{RP}^3$), the notation $\underline{x} \in \RR^3$ denotes the vector of first three components of $x$ and the notation $x_i$ stands for the i-th component of $x$. Thus, it can be written as $x = [\underline{x} \quad x_4]^\top$.

\subsection{System equations and measurements}
The vehicle's pose $X\in \text{SE(3)}$, defined by \eqref{defpose}, satisfies the kinematic equation
\begin{equation}\label{system}
\dot X = F(X,A) := X A,
\end{equation}
with group velocity $A \in \mathfrak{se}(3)$. System \eqref{system} is {\em left invariant} in the sense that it preserves the (Lie group) invariance properties with respect to constant translation and constant rotation of the body-fixed frame $\{{\cal B}\}$ $X \mapsto X_0 X$.

Assume that the group velocity $A$ (i.e. $\Omega$ and $V$) is bounded, continuous, and available to measurement. Moreover, $N \in \NN^+$ constant elements of the real projective space $\mathring{y}_i \in \mathbb{RP}^3$ ($i=1,\cdots,N$), known in the inertial frame $\{\mathcal{A}\}$, are assumed to be measured in the body-fixed frame $\{\mathcal{B}\}$ as
\begin{equation}\label{measurement}
y_i = h (X, \mathring{y}_i) := \frac{X^{-1} \mathring{y}_i}{|X^{-1} \mathring{y}_i|} \in \mathbb{RP}^3, \quad i=1,\cdots,N.
\end{equation}
Note that the Lie group action $h:  \mathrm{SE(3)}\times \mathbb{RP}^3 \rightarrow \mathbb{RP}^3$ is a  {\em right group action} in the sense that for all $X_1,X_2 \in \mathrm{SE(3)}$ and $y\in \mathbb{RP}^3$, one has $h(X_2, h(X_1,y)) = h(X_1 X_2, y)$. For later use, define
\begin{equation}\label{defy}
Y := (y_1,\cdots, y_N),\quad \mathring Y := (\mathring y_1,\cdots, \mathring y_N).
\end{equation}

\begin{rem}\label{remMeasurement}
Interestingly, by considering the measurement data in the real projective space $\mathbb{RP}^3$, we are able to combine in a sole pose observer various types of measurements coming from sensors of different nature. For instance, from a stereo camera or a Kinect sensor we can obtain a matching of $N_1 \in \NN^+$ feature points whose position coordinates are known in both the inertial reference frame $\{\mathcal{A}\}$ and the current body-fixed frame $\{\mathcal{B}\}$, i.e. one has
\[
p_i = R^\top (\mathring{p}_i - p), \quad i=1,\cdots, N_1,
\]
with $\mathring{p}_i, {p}_i \in \RR^3$ the position coordinates of the feature points expressed in the frames $\{\mathcal{A}\}$ and $\{\mathcal{B}\}$, respectively. Then, the following simple transformations:
\[
\begin{split}
&\underline{\mathring{y}}_i := \frac{\mathring{p}_i}{\sqrt{|\mathring{p}_i|^2+1}}, \quad \mathring{y}_{i,4} := \frac{1}{\sqrt{|\mathring{p}_i|^2+1}}, \\
&\underline{y}_i := \frac{p_i}{\sqrt{|p_i|^2+1}}, \quad y_{i,4} := \frac{1}{\sqrt{|{p}_i|^2+1}} ,
\end{split}
\]
yield the following relations in the form \eqref{measurement}:
\[
y_i = \frac{X^{-1} \mathring y_i}{|X^{-1} \mathring y_i|} = h(X,\mathring y_i), \quad i=1,\cdots,N_1,
\]
with $\mathring y_i =[\underline{\mathring y}_i \quad \mathring y_{i,4}]^\top \in \mathbb{RP}^3$ and $y_i =[\underline{y}_i \quad y_{i,4}]^\top \in \mathbb{RP}^3$. Such a transformation provides a ``natural'' scaling of the position measurements of known inertial feature points so that the measurement of a very far feature point will act closely to a vectorial measurement. On the other hand, assume also that the vehicle is equipped with $N_2 \in \NN^+$ vectorial sensors (e.g. magnetometer or inclinometer) so as to provide the measurements $v_j \in \RR^3$ in the body-fixed frame $\{\mathcal{B}\}$ of $N_2$ Euclidean vectors (given for example by the geomagnetic field or the gravity field) whose coordinates $\mathring v_j \in \RR^3$ in the inertial frame $\{\mathcal{A}\}$ are known. Then, one verifies that $v_j = R^\top \mathring v_j$ and deduces the following relations in the form \eqref{measurement}:
\[
y_j = \frac{X^{-1} \mathring y_j}{|X^{-1} \mathring y_j|} = h(X,\mathring y_j), \quad j=N_1+1,\cdots, N_1+N_2,
\]
with $\mathring y_j =[\underline{\mathring y}_j \quad 0]^\top \in \mathbb{RP}^3$, $y_j =[\underline{y}_j \quad 0]^\top \in \mathbb{RP}^3$, $\underline{\mathring{y}}_j := \frac{\mathring{v}_{j}}{|\mathring{v}_{j}|}$ and $\underline{y}_j := \frac{v_{j}}{|v_{j}|}$.
\end{rem}

We verify that $\mathrm{SE(3)}$ is a symmetry group with group actions $\phi: \mathrm{SE(3)} \times \mathrm{SE(3)} \longrightarrow \mathrm{SE(3)}$, $\psi: \mathrm{SE(3)}\times \mathfrak{se}(3) \longrightarrow \mathfrak{se}(3)$ and $\rho : \mathrm{SE(3)}\times \mathbb{RP}^3 \longrightarrow \mathbb{RP}^3$ defined by
\[
\begin{array}{lcl}
\phi(Q,X) &:=& X Q,\\
\psi(Q,A) &:=& Ad_{Q^{-1}} A = Q^{-1} A Q,\\
\rho(Q,y) &:=& \frac{Q^{-1} y}{|Q^{-1} y|}.
\end{array}
\]
Indeed, it is straightforward to verify that $\phi$, $\psi$, and $\rho$ are {\em right group actions} in the sense that
$\phi(Q_2, \phi(Q_1,X)) = \phi(Q_1Q_2,X)$, $\psi(Q_2, \psi(Q_1,A)) = \psi(Q_1Q_2,A)$, and $\rho(Q_2, \rho(Q_1,y)) = \rho(Q_1Q_2,y)$, for all $Q_1,Q_2,X \in \mathrm{SE(3)}$, $A\in \mathfrak{se}(3)$, and $y\in \mathbb{RP}^3$. Clearly, one has
\[
\rho(Q, h(X,\mathring{y}_i)) = \frac{Q^{-1} \frac{X^{-1} \mathring y_i}{|X^{-1} \mathring{y}_i|}}{\big|Q^{-1} \frac{X^{-1} y_i}{|X^{-1} \mathring{y}_i|}\big|}=h(\phi(Q,X),\mathring{y}_i),
\]
\[
\begin{array}{lcl}
d\phi_Q(X)[F(X,A)] &=& XA  Q = (XQ) (Q^{-1} AQ) \\
&=& F(\phi(Q,X), \psi(Q, A)).
\end{array}
\]
Thus, the kinematics \eqref{system} are {\em right equivariant} in the sense of \cite[Def. 2]{MahonyNOLCOS13}. This is a condition allowing us to apply the theory proposed in \cite{MahonyNOLCOS13} for nonlinear observer design directly on the output space. Note also that the system under consideration belongs to type I systems (see \cite{MahonyNOLCOS13}) where both the velocity sensors and the state sensors are attached to the body-fixed frame.

\section{Gradient-like observer design} \label{sec:observerDesign}
Denote by $\hat X(t)  \in  \mathrm{SE(3)}$ the estimate of the pose $X(t)$ and denote by $\hat R$ and $\hat p$ the estimates of $R$ and $p$, respectively. One has $\hat X = \begin{bmatrix} \hat R & \hat p\\ 0 & 1\end{bmatrix}$.
 Define the group error
\begin{equation}\label{righterror}
E(\hat X,X) := \hat X X^{-1} \in \mathrm{SE(3)},
\end{equation}
which is {\em right invariant} in the sense that for all $\hat X, X, Q \in \mathrm{SE(3)}$, one has $E(\hat X Q, XQ) = E(\hat X, X)$. From now on, without confusion the shortened notation $E$ is used for $E(\hat X,X)$. The group error $E$ converges to the identity element $I_4 \in \mathrm{SE(3)}$ iif $\hat X$ converges to $X$.
For later use, define also the output errors $e_i \in \mathbb{RP}^3$, with $i=1,\cdots, N$, as
\begin{equation}\label{defei}
e_i := h(\hat X^{-1}, y_i) = \frac{\hat X y_i}{|\hat X y_i|} = \frac{E \mathring y_i}{|E \mathring y_i|}.
\end{equation}
Note that $e_i$ $(i=1,\cdots, N)$ can be viewed as the estimates of $\mathring y_i$, since they converge to $\mathring y_i$ when $E$ converges to $I_4$. Note also that $e_i$ are computable by the observer.

We now proceed the observer design. As proposed by \cite{MahonyNOLCOS13}, the observer takes the form
\begin{equation}\label{systemObs}
\dot {\hat X} = \hat X A - \Delta(\hat X, Y) \hat X, \quad \hat X(0) \in \mathrm{SE(3)} ,
\end{equation}
where $\Delta(\hat X, Y) \in \mathfrak{se}(3)$, which is a matrix-valued function of $\hat X$ and $Y$ with $Y$ defined by \eqref{defy}, is the innovation term to be designed hereafter and must be {\em right equivariant} in the sense that $\forall Q \in \mathrm{SE(3)}$:
\[
\Delta(\phi(Q,\hat X), \rho(Q, Y)) = \Delta(\hat X, Y),
\]
with $\rho(Q,Y):= (\rho(Q, y_1), \cdots, \rho(Q, y_N))$.
Interestingly, if the innovation term $\Delta(\hat X, Y)$ is right equivariant, the dynamics of the group error $E$ are autonomous \cite[Th. 1]{MahonyNOLCOS13}:
\begin{equation}\label{:eq:ErrorObserverGeneral}
\dot E = -\Delta(E, \mathring{Y}) E.
\end{equation}

In order to determine the innovation term $\Delta(\hat X, Y)$, the following cost function is considered:
\begin{equation}\label{:eq:aggregateCost}
\begin{split}
  \!\!\!\!{\mathcal{C}} \colon &\mathrm{SE(3)} \!\times\! (\mathbb{RP}^3 \times \cdots \times \mathbb{RP}^3) \!\longrightarrow \mathbb{R}^+,\\
  &\quad\quad\,(\hat X,Y)\,\,\,\,\mapsto {\mathcal{C}}(\hat X,Y) :=  \!\sum_{i=1}^N\frac{k_i}{2}\left|\frac{\hat X y_i}{|\hat X y_i|} - \mathring{y}_i\right|^2
\end{split}
\end{equation}
with positive constant parameters $k_i$.
It is easily verified that the cost function ${\mathcal{C}}(\hat X,Y)$ is {\em right invariant} in the sense that ${\mathcal{C}}(\phi(Q,\hat X), \rho(Q, Y)) = {\mathcal{C}}(\hat X, Y)$ for all $Q \in \mathrm{SE(3)}$. From here, the innovation term $\Delta(\hat X, Y)$ is computed as \cite[Eq. (40)]{MahonyNOLCOS13}:
\begin{equation}\label{:eq:innoGradGeneral}
\Delta(\hat X,Y) := (\mathrm{grad}_1 {\mathcal{C}}(\hat X,Y) )\hat X^{-1},
\end{equation}
where $\mathrm{grad}_1$ is the {\em gradient} in the first variable, using a right-invariant Riemannian metric on $\mathrm{SE}(3)$.

\begin{lem}\label{:lem:innovation}
The innovation term $\Delta(\hat X,Y)$ defined by \eqref{:eq:innoGradGeneral} is right equivariant and explicitly given by
\begin{equation}\label{inno:gradRightmetric}
\Delta(\hat X, Y)  \!=\! -\mathbf{P}\!\left(\sum_{i=1}^N k_i \left(I_4 - e_i e_i^\top \right)\mathring y_i e_i^\top \right),
\end{equation}
with $e_i$ considered as functions of $\hat X$ and $y_i$, i.e. $e_i = \frac{\hat X y_i}{|\hat X y_i|}$.
\end{lem}
\begin{proof}
The proof for $\Delta(\hat X,Y)$ given by \eqref{inno:gradRightmetric} to be right equivariant is straightforward. Now, using standard rules for transformations of Riemannian gradients and the fact that the Riemannian metric is right invariant, one obtains
\begin{equation}\label{:eq:GradDerivation1}
\begin{split}
\mathcal{D}_1{\mathcal{C}}(\hat X,Y) [U \hat X]&=\langle\mathrm{grad}_1 {\mathcal{C}}(\hat X,Y), U \hat X \rangle_X \\
&= \langle \mathrm{grad}_1 {\mathcal{C}}(\hat X,Y) \hat X^{-1} \hat X, U\hat X \rangle_X\\
&= \langle \mathrm{grad}_1 {\mathcal{C}}(\hat X,Y) \hat X^{-1} , U \rangle\\
&= \langle \Delta(\hat X,Y)  , U \rangle,
\end{split}
\end{equation}
with some $U \in \mathfrak{se}(3)$. On the other hand, using \eqref{:eq:aggregateCost} one deduces
\begin{equation}\label{:eq:GradDerivation2}
\begin{array}{ll}
&\!\!\!\!\!\!\!\!\!\mathcal{D}_1{\mathcal{C}}(\hat X,Y) [U \hat X] = d_1{\mathcal{C}}(\hat X,Y)[U \hat X] \\[1ex]
& = \sum_{i=1}^N k_i \left(\frac{\hat X y_i}{|\hat X y_i|} - \mathring y_i\right)^{\!\!\top} \!\!
\left(I_4 - \frac{(\hat X y_i)(\hat X y_i)^\top}{|\hat X y_i|^2}\right) \frac{(U\hat X)y_i}{|\hat X y_i|}\\[1ex]
& = \sum_{i=1}^N k_i(e_i -\mathring y_i)^\top (I_4 -e_ie_i^\top) U e_i\\[1ex]
& = \mathrm{tr}\left(\sum_{i=1}^N k_i (I_4 -e_ie_i^\top) (e_i -\mathring y_i)e_i^\top  U^\top \right) \\[1ex]
& = \left \langle  -\sum_{i=1}^N k_i (I_4 -e_ie_i^\top) \mathring y_i e_i ^\top, U \right\rangle\\[1ex]
& = \left \langle  -\mathbf{P}\!\left(\sum_{i=1}^N k_i (I_4 -e_ie_i^\top) \mathring y_i e_i ^\top\right)\!,\, U \right\rangle.
\end{array}
\end{equation}
Finally, the expression of $\Delta(\hat X,Y)$ given by \eqref{inno:gradRightmetric} is directly obtained from \eqref{:eq:GradDerivation1} and  \eqref{:eq:GradDerivation2}.
\end{proof}

\vspace{.2cm}
Using the definition \eqref{projection} of the projection $\mathbf{P}(\cdot)$, the innovation term $\Delta(\hat X,Y)$ given by \eqref{inno:gradRightmetric} can be rewritten in matrix form as follows:
\begin{equation}\label{brokenInno}
\begin{split}
& \!\!\!\Delta(\hat X, Y) \\
& \!\!\!\!=\!\! \begin{bmatrix}\displaystyle -\frac{1}{2} \sum_{i=1}^N k_i({\underline{e}}_i \times \underline{\mathring y}_i)_{\times} & \displaystyle \sum_{i=1}^N k_i {e}_{i,4} (({{e}_i}^\top{\mathring y}_i){\underline{e}}_i - \underline{\mathring y}_i)  \\[1ex] 0& 0\end{bmatrix} \!\!\!
\end{split}
\end{equation}

\vspace{.2cm}
Using \eqref{:eq:ErrorObserverGeneral}, \eqref{:eq:innoGradGeneral} and \eqref{inno:gradRightmetric}, one deduces the error system
\begin{equation} \label{dyn:rightError}
\begin{split}
\dot{E} &=  - {\text{grad}}_1 \mathcal{C}(E,\mathring Y) \\
&= \mathbf{P}\!\left(\sum_{i=1}^N k_i \left(I_4 - e_i {e_i}^\top \right)  \mathring y_i  {e_i}^\top \right) E
\end{split}
\end{equation}
with $e_i$ considered as functions of $E$ and $\mathring y_i$, i.e. $e_i = \frac{E \mathring y_i}{|E \mathring y_i|}$.

\vspace{0.2cm}
For the sake of analysis purposes, the following assumption is introduced.
\begin{hypo}\label{hypo:observability}(Observability) The set $\{\mathring y_i \in \mathbb{RP}^3, i=1,\cdots,N\}$ satisfies one of the three following cases:
\begin{itemize}
\item Case 1 (at least 2 vectorial and 1 position measurements): There exist two different points $\mathring y_{i_1}$ and $\mathring y_{i_2}$ with $\mathring y_{i_1,4} =\mathring y_{i_2,4} = 0$ and one point $\mathring y_{j_1}$ such that $\mathring y_{j_1,4} \neq 0$.
\item Case 2 (at least 1 vectorial and 2 position measurements): There exist one point $\mathring y_{i_1}$ with $\mathring y_{i_1,4} = 0$ and two different points $\mathring y_{j_1}$ and $\mathring y_{j_2}$  (i.e., $\mathring y_{j_1} \neq \mathring y_{j_2}$) with $\mathring y_{j_1,4} \neq 0$ and $\mathring y_{j_2,4} \neq 0$. Furthermore, the vector $\mathring {\underline y}_{i_1}$ and the resultant vector $v_{j_{12}} := \mathring y_{j_2,4}\, \mathring {\underline y}_{j_1}-\mathring y_{j_1,4}\,\mathring {\underline y}_{j_2}$ are non-collinear.
\item Case 3 (at least 3 position measurements): There exist three different points $\mathring y_{j_1}$, $\mathring y_{j_2}$ and $\mathring y_{j_3}$ such that $\mathring y_{j_1,4} \neq 0$, $\mathring y_{j_2,4} \neq 0$ and $\mathring y_{j_3,4} \neq 0$. Furthermore, the resultant vectors $v_{j_{12}}:= \mathring y_{j_2,4}\, \mathring {\underline y}_{j_1}-\mathring y_{j_1,4}\,\mathring {\underline y}_{j_2}$, $v_{j_{23}}:= \mathring y_{j_3,4}\, \mathring {\underline y}_{j_2}-\mathring y_{j_2,4}\,\mathring {\underline y}_{j_3}$ and $v_{j_{31}}:= \mathring y_{j_1,4}\, \mathring {\underline y}_{j_3}-\mathring y_{j_3,4}\,\mathring {\underline y}_{j_1}$ are not all collinear.
\end{itemize}
\end{hypo}

From here, the first result of this paper is stated.
\begin{theo}\label{theorem:StabilityS3}
Consider the kinematics \eqref{system}. Consider the observer \eqref{systemObs} with the innovation term $\Delta(\hat X,Y)$ given by \eqref{inno:gradRightmetric}.
Assume that Assumption \ref{hypo:observability} is satisfied. Then, the equilibrium $E= I_4$ of the error system \eqref{dyn:rightError} is locally asymptotically stable.
\end{theo}
\begin{proof}
Since the right-hand side of \eqref{dyn:rightError} is a gradient flow of $\mathcal{C}$, in order to prove the local asymptotic stability of $E = I_4$, it suffices to prove that $\mathcal{C}(E, \mathring Y)$ is minimal when $E = I_4$.
Note that
\begin{equation}\label{costfunctionbis}
\mathcal{C}(E, \mathring Y) = {\cal V}(E) := \frac{1}{2} \sum_{i=1}^N k_i \left|\frac{E \mathring y_i}{|E \mathring y_i|} - \mathring y_i\right|^2.
\end{equation}
Let us prove that the function ${\cal V}(E)$ has a unique global minimum at $E = I_4$, {i.e.}
${\cal V}(E) = 0  \Leftrightarrow  E = I_4.$\\
First, it is straightforward to verify that ${\cal V}(I_4) = 0$. Denote
$E = \begin{bmatrix} R_e & p_e \\0 & 1 \end{bmatrix}$, with $R_e \in \mathrm{SO(3)}, p_e \in \mathbb{R}^3$. Now assuming that ${\cal V}(E) = 0$, we only have to prove that $E = I_4$ or,  equivalently, $R_e = I_3$ and $p_e= 0$. In view of \eqref{costfunctionbis} and ${\cal V}(E) = 0$, one deduces that
$ E \mathring y_i = |E\mathring y_i| \mathring y_i$, $\forall i$, {i.e.}
\begin{subequations}\label{eqErRepe}
\begin{empheq}[left=\empheqlbrace]{align}
R_e \mathring{\underline y}_i + p_e \mathring y_{i,4} &= \sqrt{\mathring y_{i,4}^2 + |R_e \mathring{\underline y}_i + p_e \mathring y_{i,4}|^2} \,\,\mathring{\underline y}_i \label{eqErRepe1}\\
\mathring y_{i,4} &= \sqrt{\mathring y_{i,4}^2 + |R_e \mathring{\underline y}_i + p_e \mathring y_{i,4}|^2} \,\, \mathring y_{i,4} \label{eqErRepe2}
\end{empheq}
\end{subequations}
Let us consider all the three cases of Assumption \ref{hypo:observability}.

\vspace{0.1cm}
\noindent $\bullet$ {{\em Case 1 of Assumption \ref{hypo:observability}:}} Since $\mathring y_{i_1,4} =\mathring y_{i_2,4} = 0$, one has $|\mathring {\underline y}_{i_1}| = |\mathring {\underline y}_{i_2}| = 1$. Then, one deduces from \eqref{eqErRepe1} that $R_e \mathring{\underline y}_{i_1} = \mathring{\underline y}_{i_1}$ and $R_e \mathring{\underline y}_{i_2} = \mathring{\underline y}_{i_2}$. These equalities and the non-collinearity of $\mathring{\underline y}_{i_1}$ and $\mathring{\underline y}_{i_2}$ allows one to deduce that $R_e = I_3$. Since $\mathring y_{j_1,4} \neq 0$, \eqref{eqErRepe2} implies that $|E \mathring y_{j_1} | =1$. As a consequence, one deduces from \eqref{eqErRepe1} that $p_e = 0$.

\vspace{0.1cm}
\noindent $\bullet$ {{\em Case 2 of Assumption \ref{hypo:observability}:}} Analogously to case 1, one deduces that $R_e \mathring{\underline y}_{i_1} = \mathring{\underline y}_{i_1}$. Now, since $\mathring y_{j_1,4} \neq 0$ and $\mathring y_{j_2,4} \neq 0$,  \eqref{eqErRepe2} implies that $|E \mathring y_{j_1} | = |E \mathring y_{j_2} |=1$. Then, from \eqref{eqErRepe1} one obtains
\begin{subequations}
\begin{empheq}[left=\empheqlbrace]{align}
(R_e - I_3)\mathring{\underline y}_{j_1} + p_e \mathring y_{j_1,4} &= 0 \notag\\
(R_e - I_3) \mathring{\underline y}_{j_2} + p_e \mathring y_{j_2,4} &= 0 \notag
\end{empheq}
\end{subequations}
From here, simple combination yields $R_e v_{j_{12}} = v_{j_{12}}$, with $v_{j_{12}}$ defined in Assumption \ref{hypo:observability}.
It is easily verified that $v_{j_{12}} \neq 0$ using the fact that $\mathring y_{j_1}$ and $\mathring y_{j_2}$ are non-collinear by assumption. Furthermore, since $\mathring {\underline y}_{i_1}$ and $v_{j_{12}}$ are non-collinear by assumption, relations $R_e \mathring{\underline y}_{i_1} = \mathring{\underline y}_{i_1}$ and $R_e v_{j_{12}} = v_{j_{12}}$ obtained previously imply that $R_e = I_3$. From here, it is straightforward to deduce that $p_e = 0$.

\vspace{0.1cm}
\noindent $\bullet$ {{\em Case 3 of Assumption \ref{hypo:observability}:}} Analogously to case 2, one deduces from \eqref{eqErRepe} that $|E \mathring y_{j_1}| = |E \mathring y_{j_2}|=|E \mathring y_{j_3}|=1$ and
\begin{subequations}
\begin{empheq}[left=\empheqlbrace]{align}
(R_e - I_3) \mathring{\underline y}_{j_1} + p_e \mathring y_{j_1,4} &= 0 \notag\\
(R_e - I_3) \mathring{\underline y}_{j_2} + p_e \mathring y_{j_2,4} &= 0 \notag\\
(R_e - I_3) \mathring{\underline y}_{j_3} + p_e \mathring y_{j_3,4} &= 0 \notag
\end{empheq}
\end{subequations}
From here, analogously to case 2 one deduces that $R_e v_{j_{12}} = v_{j_{12}}$, $R_e v_{j_{23}} = v_{j_{23}}$, $R_e v_{j_{31}} = v_{j_{31}}$, and that $v_{j_{12}}$, $v_{j_{23}}$ and $v_{j_{31}}$ are not null. Then, using the non-collinearity assumption of the vectors $v_{j_{12}}$, $v_{j_{23}}$ and $v_{j_{31}}$, it is easily deduced that $R_e = I_3$ and, consequently, that $p_e= 0$.
\end{proof}

\section{Observer design with velocity bias compensation}\label{sec:observerDesignBias}
The observer developed in the previous section will be extended in order to cope with the case where the measurement $A_y\in \mathfrak{se}(3)$ of the group velocity $A\in \mathfrak{se}(3)$ is corrupted by an unknown constant bias $b_A \in \mathfrak{se}(3)$, i.e. $A_y = A + b_A$.

\vspace{-0.1cm}
\begin{hypo}\label{hypo:observability_bias}
Assume that the following matrices $\mathring G \in \RR^{3\times 3}$ and $\mathring H \in \RR^{3\times 3}$ are full rank:
\[
\begin{split}
\mathring G := &\sum_{i=1}^N k_i (\mathring{\underline{y}}_{i \times})^2  \\
\mathring H := &\left(\sum_{i=1}^N k_i\mathring y_{i,4} \mathring{\underline{y}}_{i\times}\right) \mathring G^{-1} \left(\sum_{i=1}^N k_i\mathring y_{i,4} \mathring{\underline{y}}_{i\times}\right) \\
& - \sum_{i=1}^N k_i \mathring y_{i,4}^2 (I_3 - \mathring{\underline{y}}_i \mathring{\underline{y}}_i^\top )
\end{split}
\]
\end{hypo}

\vspace{-0.1cm}
The condition on the set $\{\mathring y_i \in \mathbb{RP}^3, i=1,\cdots,N\}$ evoked in Assumption \ref{hypo:observability} ensures that it is always possible to choose a set of parameters $\{k_i, i=1,\cdots,N\}$ such that $\mathring G$ and $\mathring H$ are full rank (i.e. invertible). Now, the second result of this paper is stated.

\vspace{-0.1cm}
\begin{propo}\label{theorem:Stability_bias}
Consider the observer system
\begin{subequations}\label{observer:Right_bias}
\begin{empheq}[left=\empheqlbrace]{align}
&\dot {\hat X}  =  \hat X(A_y - \hat b_A) - \Delta(\hat X,Y) \hat X \label{observer:Right_biasEq1}\\
&\dot {\hat b}_A =  -k_b\mathbf{P}\!\left(\hat X^\top\! \sum_{i=1}^N k_i \left(I_4 \!-\! e_i {e_i}^\top \right) \mathring y_i  {e_i}^\top \!\hat X^{-\top} \!\right) \label{observer:Right_biasEq2}\\
&\hat X(0) \in \mathrm{SE(3)}, \,\,\, \hat b_A(0) \in \mathfrak{se}(3) \notag
\end{empheq}
\end{subequations}
with $\Delta(\hat X,Y)$ given by \eqref{inno:gradRightmetric}. Assume that Assumptions \ref{hypo:observability} and \ref{hypo:observability_bias} are satisfied. Assume also that $A$ and $X$ are bounded for all time.~Then, the equilibrium $(E, \tilde b_A) = (I_4, 0)$ of the dynamics of $(E, \tilde b_A)$, with $\tilde b_A := b_A - \hat b_A$, is locally asymptotically stable.
\end{propo}
\vspace{-0.1cm}
\begin{proof} It is easily verified that $\dot{\tilde b}_A = -\dot{\hat b}_A$ and
$\dot {\hat X} = \hat X(A + \tilde b_A) - \Delta(\hat X,Y) \hat X$. Then, one deduces
\begin{equation} \label{dyn:rightErrorS3_bias}
\dot{E} = \left(Ad_{\hat X}\tilde b_A - \Delta(E,\mathring Y)\right) E.
\end{equation}
Consider the candidate Lyapunov function
\begin{equation}\label{costfunctionbisformS3_bias}
{\cal V}_b(E,\tilde b_A) := \frac{1}{2} \sum_{i=1}^N k_i \left|\frac{E \mathring y_i}{|E \mathring y_i|} - \mathring y_i\right|^2 + \frac{1}{2k_b}\|\tilde b_A\|^2.
\end{equation}
Analogously to the proof of Theorem \ref{theorem:StabilityS3}, it can be verified that ${\cal V}_b(E,\tilde b_A)$ is locally positive-definite and has a unique global minimum at $(E,\tilde b_A) =(I_4,0)$, {i.e.} ${\cal V}_b(E,\tilde b_A) = 0  \Leftrightarrow  (E,\tilde b_A) =(I_4,0)$.

The time-derivative of ${\cal V}_b$ satisfies
\begin{equation}\label{dotcostfunctionbisformS3_bias}
\begin{array}{ll}
\dot{\cal V}_b  &\!\! =\left\langle -\sum_{i=1}^N k_i \left(I_4 \!-\! {e}_i {{e}_i}^\top \right)
\mathring y_i {{e}_i}^\top\!,  Ad_{\hat X}\tilde b_A -\Delta(E,\mathring Y)\right\rangle \\
& \quad\! - \frac{1}{k_b}\left\langle \dot {\hat b}_A, \tilde b_A \right\rangle\\[1ex]
&  \!\! = -\left\|\mathbf{P}\left(\sum_{i=1}^N k_i \left(I_4 - e_i e_i^\top \right)  \mathring y_i  e_i^\top \right)\right\|^2  \\[1ex]
&  \!\!= - \|\Delta(E,\mathring Y)\|^2. \!\!\!\!
\end{array} \!\!\!\!
\end{equation}
Since the dynamics of $(E, \tilde b_A)$ are not autonomous, LaSalle's theorem does not apply to deduce the convergence of $\dot{\cal V}_b$ to zero. Thus, the next step of the proof consists in proving that $\dot{\cal V}_b$ is (locally) uniformly continuous in order to deduce, by application of Barbalat's lemma, the convergence of $\dot{\cal V}_b$ to zero. To this purpose it suffices to prove that $\ddot{\cal V}_b$ is bounded. In view of \eqref{dotcostfunctionbisformS3_bias}, $\ddot{\cal V}_b$ is bounded if $\dot {e}_i$ ($i=1,\cdots, N$) are bounded, where (using  \eqref{dyn:rightErrorS3_bias} and the relation $e_i = \frac{E\mathring y_i}{|E\mathring y_i|}$)
\[
\dot {e}_i = (I_4 - e_i e_i^\top) (Ad_{\hat X}\tilde b_A - \Delta(E,\mathring Y)) e_i.
\]
According to Assumption \ref{hypo:observability_bias}, there exists at least one point $\mathring y_i$ such that its fourth component $\mathring y_{i,4}$ is not null. This indicates that for a given small number $\eps >0$ there exists $\delta_\eps >0$ such that if $|p_e|>\delta_\eps$ or $|\tilde b_A|>\delta_\eps$ then ${\cal V}_b(E,\tilde b_A) > \eps$. Therefore, there exists a small enough neighborhood $\mathfrak{B}_{\eps} \in \mathrm{SE(3)} \times \RR^3$ of the point $(I_4, 0)$ such that if $(E(0), \tilde b_A(0)) \in \mathfrak{B}_\eps$ then ${\cal V}_b(E(0),\tilde b_A(0)) < \eps$. Since ${\cal V}_b(E,\tilde b_A)$ is non-increasing, one has ${\cal V}_b(E(t),\tilde b_A(t)) < \eps, \forall t\leq 0$. This implies that $E$ and $\tilde b_A$ remain bounded. Since $X$ is bounded by assumption, one deduces from the boundedness of $E$ that $\hat X$ is also bounded, which in turn implies the boundedness of $\dot{E}$ and $\dot {{e}}_i$. This concludes the proof of (local) uniform continuity of $\dot{\cal V}_b$ and the convergence of $\dot{\cal V}_b$ to zero. One easily verifies that $(E,\tilde b_A) =(I_4,0)$ is an equilibrium of the error system. Let us prove the local stability of this equilibrium. To this purpose let us first prove that $\forall (E, \tilde b_A) \in \mathfrak{B}_\eps$: \vspace{-0.cm}
\[
\left\{
\begin{array}{lcl}
\dot{\cal V}_b(E,\tilde b_A) = 0 &\text{if}& E = I_4 \\
\dot{\cal V}_b(E,\tilde b_A) < 0 &\text{if}& E \neq I_4 \\
\end{array}
\right. \vspace{-0.2cm}
\]
Consider a first order approximation of $E =  \begin{bmatrix}R_e & p_e \\ 0 & 1 \end{bmatrix}$ around $I_4$ as
\[
\left\{
\begin{array}{lcl}
p_e &=& \eps_p \\
R_e &=& I_3 + \eps_{r\times}
\end{array}
\right.
\]
with $\eps_p, \eps_r \in \RR^3$. We only need to prove that
\[
\dot{\cal V}_b(E,\tilde b_A) = 0 \,\, \Leftrightarrow \,\, \eps_p =\eps_r = 0.
\]
Note that \eqref{dotcostfunctionbisformS3_bias} and \eqref{brokenInno} indicate that the relation $\dot{\cal V}_b = 0$ is equivalent to
\begin{equation}\label{S3biasConditionsApprox}
\left\{
\begin{array}{lcl}
\sum_{i=1}^N k_i \underline {e}_i \times \underline{\mathring y}_i  &=& 0 \\
\sum_{i=1}^N k_i(({{e}_i}^\top{\mathring y}_i)\underline {{e}}_i - \underline{\mathring y}_i){e}_{i,4} &=& 0
\end{array}
\right.
\end{equation}
In first order approximations, one verifies that
\[
E \mathring y_i = \begin{bmatrix} \mathring {\underline y}_{i} + \eps_{r\times} \mathring {\underline y}_{i} + \mathring y_{i,4}\eps_p \\ \mathring y_{i,4}\end{bmatrix}, \quad
|E \mathring y_i| = 1 + \mathring y_{i,4} \eps_p^\top\mathring {\underline y}_{i},
\]
and, thus,
\[
{{e}}_i =\frac{E \mathring y_i}{|E \mathring y_i|} = \begin{bmatrix} \mathring {\underline y}_{i} + \eps_{r\times} \mathring {\underline y}_{i} + \mathring y_{i,4}(I_3 - \mathring {\underline y}_{i}\mathring {\underline y}_{i}^\top)\eps_p \\[1ex] \mathring y_{i,4} - \mathring y_{i,4}^2 \eps_p^\top\mathring {\underline y}_{i}\end{bmatrix}.
\]
Therefore, in first order approximations the equalities in \eqref{S3biasConditionsApprox} can be rewritten as
\begin{subequations}\label{S3biasConditionsApproxbis}
\begin{empheq}[left=\empheqlbrace]{align}
&\!\!\left(\sum_{i=1}^N k_i\mathring y_{i,4} \mathring{\underline{y}}_{i\times}\right)\eps_p  = \left(\sum_{i=1}^N k_i (\mathring{\underline{y}}_{i \times})^2\right)\eps_r  \label{S3biasConditionsApproxbis1}\\
&\!\!\left(\!\sum_{i=1}^N k_i \mathring y_{i,4}^2 (I_3 - \mathring{\underline{y}}_i \mathring{\underline{y}}_i^\top)\!\right)\! \eps_p \!=\! \left(\sum_{i=1}^N k_i\mathring y_{i,4} \mathring{\underline{y}}_{i\times}\right)\eps_r \!\! \label{S3biasConditionsApproxbis2}
\end{empheq}
\end{subequations}
Since $\mathring G$ is full rank according to Assumption \ref{hypo:observability_bias}, it is deduced from \eqref{S3biasConditionsApproxbis1} that $\eps_r = \mathring G^{-1}\left(\sum_i k_i\mathring y_{i,4} \mathring{\underline{y}}_{i\times}\right)\eps_p$. This relation along with \eqref{S3biasConditionsApproxbis2} yields $\mathring H \eps_p = 0$. Since $\mathring H$ is also full rank by Assumption \ref{hypo:observability_bias}, it is deduced that $\eps_p = 0$ and, consequently, $\eps_r = 0$.

It remains to prove to convergence of $\tilde b_A$ to zero. From the convergence of $E$ to $I_4$ (proven previously) and  \eqref{dyn:rightErrorS3_bias}, the application of Barbalat's lemma yields the convergence of $\dot E$ to zero. Finally, Eq. \eqref{dyn:rightErrorS3_bias} and the convergence of $\dot E$ and of $\Delta(E,\mathring Y)$ to zero imply the convergence of $\tilde b_A$ to zero.
\end{proof}

\vspace{0.1cm}
The estimate $\hat b_A$ plays the role of integral correction for the error dynamics \eqref{dyn:rightErrorS3_bias}, allowing for the compensation of the unknown constant bias $b_A$. It may, however, grow arbitrarily large, resulting in slow convergence and sluggish dynamics of the error evolution. This leads us to replace hereafter the integral term $\hat b_A$, with dynamics given by \eqref{observer:Right_biasEq2}, by an ``anti-windup'' integrator similar to the one proposed in \cite{hua14,huasamson11}. More precisely, by decomposing $\hat b_A$ as $\hat b_A = \begin{bmatrix} (\hat b_\Omega)_\times & \hat b_V \\ 0 & 0 \end{bmatrix}$ with $\hat b_\Omega, \hat b_V \in \RR^3$, one rewrites the dynamics \eqref{observer:Right_biasEq2} of the estimated bias $\hat b_A$ as
\begin{subequations}\label{biasDyn}
\begin{empheq}[left=\empheqlbrace]{align}
\dot{\hat b}_\Omega &= k_b \hat R^\top \!\!\left(\Omega_\Delta  + \frac{1}{2} V_\Delta \times \hat p\right)  \notag\\
\dot{\hat b}_V &= k_ b \hat R^\top V_\Delta \notag
\end{empheq}
\end{subequations}
with $V_\Delta := \sum_{i=1}^N k_i {e}_{i,4} (({{e}_i}^\top{\mathring y}_i){\underline{e}}_i - \underline{\mathring y}_i)$ and $\Omega_\Delta := -\frac{1}{2} \sum_{i=1}^N k_i {\underline{e}}_i \times \underline{\mathring y}_i$.
From here, the following modified dynamics of  $\hat b_A$ (i.e. of $\hat b_\Omega$ and $\hat b_V$) are proposed: \vspace{-0.1cm}
\begin{subequations}\label{biasDynModified}
\!\!\!\!\begin{empheq}[left=\empheqlbrace]{align}
\dot{\hat b}_\Omega &=\! k_b \hat R^\top \!(\Omega_\Delta  \!\!+\!\! \frac{1}{2} V_\Delta\!\! \times \hat p)
 - \kappa_\Omega (\hat b_\Omega - \mathrm{sat}_{\delta_\Omega}(\hat b_\Omega)) \label{biasDynModifieda}\\
\dot{\hat b}_V &= k_ b \hat R^\top V_\Delta  - \kappa_V (\hat b_V - \mathrm{sat}_{\delta_V}(\hat b_V))\label{biasDynModifiedb}
\end{empheq}\!\!
\end{subequations}
with initial conditions $|b_\Omega(0)| \leq \delta_\Omega$ and $|b_V(0)|\leq \delta_V$; $\kappa_\Omega$ and $\kappa_V$ two positive numbers; $\delta_\Omega$ and $\delta_V$ two positive parameters associated with the classical functions $\mathrm{sat}_{\delta_\Omega}(\cdot)$ and $\mathrm{sat}_{\delta_\Omega}(\cdot)$ defined by $\mathrm{sat}_{\delta}(x) = x \min(1,\delta/|x|), \forall x\in \RR^3$.

\vspace{-0.1cm}
\begin{coro}\label{coro:biasAntiwindup}
Consider the observer \eqref{observer:Right_biasEq1}+\eqref{biasDynModified}. Assume that $\delta_\Omega$ and $\delta_V$ are chosen such that $|b_\Omega| \leq \delta_\Omega$ and $|b_V| \leq \delta_V$. Then, provided that all the assumptions of Proposition \ref{theorem:Stability_bias} are satisfied, the local asymptotic stability property of Proposition \ref{theorem:Stability_bias} holds.
\end{coro}
\vspace{-0.2cm}
\begin{proof}
Based on the inequality $|b - \mathrm{sat}_\delta(b-\tilde b)|\leq |\tilde b|$ for all $\tilde b\in \RR^3$ and provided that $\delta \geq |b|$ (see e.g. \cite{hua14}), it can be easily proved that the time-derivative of $\mathcal V_b$ defined by \eqref{costfunctionbisformS3_bias} satisfies $\dot{\mathcal V}_b \leq - \|\Delta(E,\mathring Y)\|^2$. Therefore, the local asymptotic stability property given in Proposition \ref{theorem:Stability_bias} still holds when the dynamics of $\hat b_A$ given by \eqref{observer:Right_biasEq2} is replaced by \eqref{biasDynModified}.
\end{proof}

\section{Simulation results}\label{sec:simulation}
In this section, the performance of observer  \eqref{observer:Right_biasEq1}+\eqref{biasDynModified} is illustrated by simulations. The angular and translational velocity measurements are corrupted by the following constant biases:
\[\begin{split}
b_\Omega& = [-0.02 \quad 0.02 \quad 0.01]^\top \quad\mathrm{(rad/s)},\\
b_V &=[0.2 \quad -0.1 \quad 0.1]^\top \quad\mathrm{(m/s)}.
\end{split}
\]
We consider the three following cases where only three system outputs $y_i \in \mathbb{RP}^3$ of known inertial elements $\mathring y_i \in \mathbb{RP}^3$ ($i=1,2,3$) are available to measurement:
\begin{itemize}
\item Case 1: corresponds to Case 1 of Assumption \ref{hypo:observability}, in which two vectorial measurements $v_{1}, v_2 \in \RR^3$ and the position measurement $p_1 \in \RR^3$ of one feature point are available, where
    $v_1 = R^\top \mathring v_1$, $v_2 = R^\top \mathring v_2$, $p_1 = R^\top (\mathring p_1 - p)$, with $\mathring v_1 = [0\quad 0 \quad 1]^\top$, $\mathring v_2 = [\sqrt{3}/2\quad 1/2 \quad 0]^\top$ and $\mathring p_1 = [1\quad 0\quad 0]^\top$.
\item Case 2: corresponds to Case 2 of Assumption \ref{hypo:observability}, in which one vectorial measurement $v_{1} \in \RR^3$ and the position measurements $p_1, p_2 \in \RR^3$ of two feature points are available, where
    $v_1 = R^\top \mathring v_1$, $p_1 = R^\top (\mathring p_1 - p)$, $p_2 = R^\top (\mathring p_2 - p)$,
    with $\mathring v_1 = [0\quad 0 \quad 1]^\top$, $\mathring p_1 = [1\quad 0\quad 0]^\top$  and $\mathring p_2 = [-1/2\quad \sqrt{3}/2 \quad 0]^\top$.
\item Case 3: corresponds to Case 3 of Assumption \ref{hypo:observability}, in which the position measurements $p_1, p_2, p_3 \in \RR^3$ of three feature points are available, where
    $p_1 = R^\top (\mathring p_1 - p)$, $p_2 = R^\top (\mathring p_2 - p)$, $p_3 = R^\top (\mathring p_3 - p)$,
    with $\mathring p_1 = [1\quad 0\quad 0]^\top$, $\mathring p_2 = [-1/2\quad \sqrt{3}/2 \quad 0]^\top$ and $\mathring p_3 = [-1/2\quad -\sqrt{3}/2 \quad 0]^\top$.
\end{itemize}
Recall that Remark \ref{remMeasurement} explains how to transform a vector or a position of a feature point into a corresponding element of $\mathbb{RP}^3$.

The gains and parameters involved in the proposed observer are chosen as follows:
\[
\begin{split}
&k_1 = k_2 = k_3 = 2, \, k_b = 1, \\
&\kappa_\Omega=\kappa_V = 10, \, \delta_\Omega = 0.052,\, \delta_V =0.346.
\end{split}
\]
For each simulation run, the proposed filter is initialized at the origin (i.e. $\hat R=I_3, \hat p=0, \hat b_\Omega=0, \hat b_V =0$) while the true trajectories are initialized differently. Combined sinusoidal inputs are considered for both the angular and translational velocity inputs of the system kinematics. The rotation angle associated with the axis-angle representation is used to represent the attitude trajectory. One can observe from Figure \ref{fig1} that the observer trajectories converge to the true trajectories after a short transition period for all the three considered cases. Figure \ref{fig2} shows that the norms of the estimated velocity bias errors $|\tilde b_\Omega|$ and $|\tilde b_V|$ converge to zero, which means that the group velocity bias $b_A$ is also correctly estimated.

\begin{figure}[!t]\centering%
\psfrag{t (s)}{\scriptsize $t (s)$}%
\psfrag{x (m)}{\scriptsize $x (m)$}%
\psfrag{y (m)}{\scriptsize $y (m)$}%
\psfrag{z (m)}{\scriptsize $z (m)$}%
\psfrag{attitude (rad)}{\scriptsize $\text{Rot. angle (rad)}$}%
\includegraphics[width=\linewidth, height = 8cm]{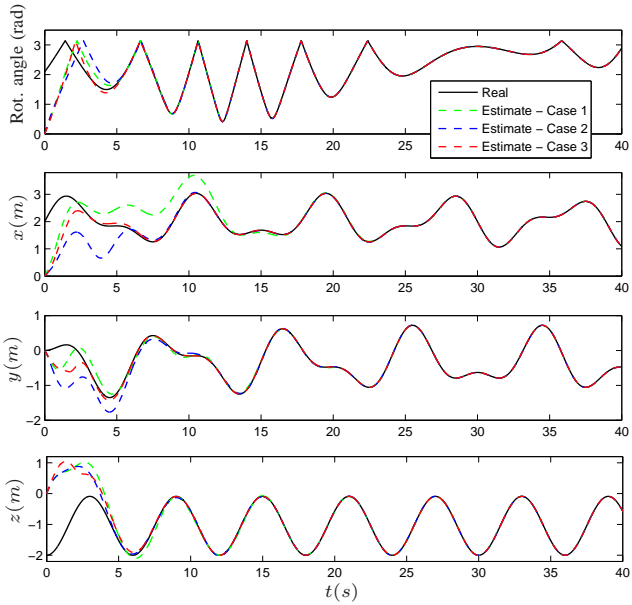} \vspace*{-0.7cm}
\caption{The rotation angle and the position tracking performance of the
proposed observer. Note that the dashed lines are the estimated trajectories (for Cases 1 (green), Case 2 (blue), Case 3 (red)) while
the solid line (black) represents the true trajectory.} \label{fig1} \vspace{-0.4cm}
\end{figure}

\begin{figure}[!t]\centering%
\psfrag{t (s)}{\scriptsize $t (s)$}%
\psfrag{norm bias Omega}{\scriptsize $|\tilde b_\Omega|$}%
\psfrag{norm bias V}{\scriptsize $|\tilde b_V|$}%
\includegraphics[width=\linewidth]{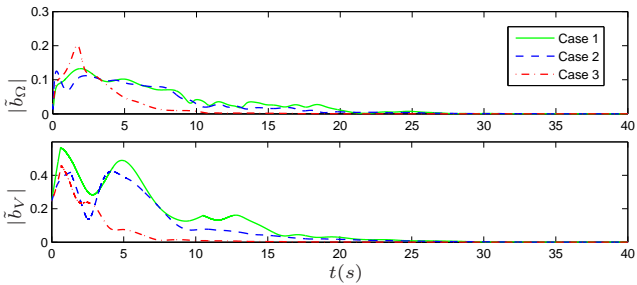} \vspace*{-0.7cm}
\caption{The norms of the estimated velocity bias errors $|\tilde b_\Omega|$ and $|\tilde b_V|$ vs. time.} \label{fig2} \vspace{-0.4cm}
\end{figure}

\section{Conclusions}\label{sec:conclusions}
In this paper, we propose a nonlinear observer on $SE(3)$ for full pose estimation that takes the system outputs on the real projective space $\mathbb{RP}^3$ directly as inputs. The observer derivation is based on a recent observer design technique directly on the output space, proposed in \cite{MahonyNOLCOS13}. An advantage with respect to our prior work \cite{hua11SE3} is that we can now incorporate in a unique observer different types of measurements such as vectorial measurements of known inertial vectors and position measurements of known feature points. The proposed observer is also extended on $\mathrm{SE(3)} \times \mathfrak{se}(3)$ so as to compensate for unknown additive constant bias in the velocity measurements. Rigorous stability analyses are equally provided. Excellent performance of the proposed observers are justified through simulations.

\section*{Acknowledgement}
This work was supported by the French {\em Agence Nationale de la Recherche} through the ANR ASTRID SCAR project ``Sensory Control of Aerial Robots'' (ANR-12-ASTR-0033) and the Australian Research Council through the ARC Discovery Project DP120100316 ``Geometric Observer Theory for Mechanical Control Systems''.


\end{document}